\documentclass[11pt,reqno]{amsart}

\usepackage[linktocpage=true]{hyperref}
\hypersetup{colorlinks,linkcolor=blue,urlcolor=blue,citecolor=blue}
\usepackage{amssymb}
\usepackage[cmtip,all]{xy}
\usepackage{enumerate} 
\usepackage{mathdots} 
\usepackage{blkarray} 

\usepackage{dsfont}
\usepackage[mathscr]{euscript}

\usepackage{amsmath,amssymb,amscd, todonotes,appendix}
\usepackage{graphics,verbatim}
\usepackage{todonotes}

\usepackage[shortlabels]{enumitem}
\usepackage{hyperref}
\usepackage{setspace} 

\usepackage{xcolor}
\usepackage{soul}

\usepackage{float} 
\usepackage[labelformat=empty]{caption}
\usepackage{fancybox} 

\usepackage{extarrows}

\usepackage{bbm}

\usepackage{verbatim}

\usepackage{multirow}

\usepackage{tikz}
\usetikzlibrary{cd,arrows}

\usepgflibrary{shapes.geometric}

\tikzstyle{V}=[draw, fill =black, circle, inner sep=0pt, minimum size=1.5pt]
\tikzstyle{C}=[draw, fill =white, circle, inner sep=0pt, minimum size=1.5pt]
\tikzstyle{over}=[draw=white,double=black,line width=2pt, double distance=.5pt]
\numberwithin{equation}{section}
\usepackage[margin = 0.9in]{geometry}
\theoremstyle{definition}

\newtheorem*{mainthm}{Theorem A}

\newtheorem*{maincor}{Corollary B}

\newtheorem{theorem}{Theorem}[subsection]

\makeatletter\let\c@fact\c@theorem\makeatother

\makeatletter\let\c@note\c@theorem\makeatother

\makeatletter\let\c@lemma\c@theorem\makeatother

\newtheorem{assumption}[theorem]{Assumption}
\makeatletter\let\c@lemma\c@theorem\makeatother

\makeatletter\let\c@quest\c@theorem\makeatother

\makeatletter\let\c@prop\c@theorem\makeatother

\newtheorem{conjecture}[theorem]{Conjecture}

\makeatletter\let\c@conj\c@theorem\makeatother

\newtheorem{corollary}[theorem]{Corollary}
\makeatletter\let\c@conj\c@theorem\makeatother

\makeatletter\let\c@cor\c@theorem\makeatother

\newtheorem{definition}[theorem]{Definition}
\makeatletter\let\c@defn\c@theorem\makeatother

\theoremstyle{definition}

\newtheorem{remark}{Remark}[subsection]
\makeatletter\let\c@remark\c@theorem\makeatother


\def\<{\langle}
\def\>{\rangle}

\newcommand{\Ann}{\mathrm{Ann}}

\newcommand{\Ext}{\text{Ext}}

\newcommand{\fg}{\mathfrak{g}}
\newcommand{\fz}{\mathfrak{z}}

\newcommand{\Lie}{\text{Lie}}

\newcommand{\stab}{\textup{stab}}

\newcommand{\0}{\bar 0}

\newcommand{\losemi}{{\otimes \kern -.78em \ltimes}}
\newcommand{\rosemi}{{\otimes \kern -.78em \rtimes}}

\newcommand{\Hom}{\ensuremath{\operatorname{Hom}}}
\newcommand{\End}{\ensuremath{\operatorname{End}}}

\newcommand{\glmn}{\mathfrak{gl}(m|n)}

\newcommand{\Add}{\operatorname{Add}}

\newcommand{\Loc}{\operatorname{Loc}}
\newcommand{\Proj}{\operatorname{Proj}}

\newcommand{\bT}{\mathbf T}
\newcommand{\bK}{\mathbf K}

\newcommand{\unit}{\ensuremath{\mathbbm{1}}}

\newcommand{\Mod}{\operatorname{Mod}}

\newcommand{\Spc}{\operatorname{Spc}}

\newcommand{\Spec}{\operatorname{Spec}}
\newcommand{\Stab}{\operatorname{Stab}}
\newcommand{\Inj}{\operatorname{Inj}}

\newcommand{\supp}{\operatorname{supp}}
\newcommand{\Supp}{\operatorname{Supp}}

\newcommand{\opH}{\operatorname{H}}
\newcommand{\opExt}{\operatorname{Ext}}

\makeatletter
\newcommand{\leqnomode}{\tagsleft@true}
\newcommand{\reqnomode}{\tagsleft@false}
\makeatother

\allowdisplaybreaks

\title{On localizing subcategories of Lie superalgebra representations}
\author{Matthew H. Hamil}

\address{Department of Mathematics, University of Georgia, Athens, GA 30602, USA}
\email{matthew.hamil25@uga.edu}
\begin{document}

\begin{abstract}

We state and prove a stratification result that allows us to classify the tensor ideal localizing subcategories for the stable module category $\Stab(\mathcal{C}_{(\fg, \fg_{\bar 0})})$ of Lie superalgbera representations which are semisimple as representations of $\mathfrak{g}_{\bar 0}$ under the hypotheses that $\mathfrak{g}$ is a classical Lie superalgebra with a splitting detecting subalgebra $\fz \leq\fg$, as well as a natural hypothesis on realization of supports. This extends the work of the author and Nakano in \cite{HN24} where a similar classification was obtained for the stable category of modules over a detecting subalgebra employing stratification in the sense of Benson, Iyengar, and Krause \cite{BIK}. Our new result involves making use of a more general stratification framework in weakly Noetherian contexts developed by Barthel, Heard, and Sanders in \cite{BHS21} using the Balmer-Favi notion of support for big objects in tensor triangulated categories, as well as the recently developed homological stratification of Barthel, Heard, Sanders, and Zou in \cite{BHSZ24} using the homological spectrum. 

\end{abstract}
 
\maketitle 


\section{Introduction}

\subsection{}Let $\mathfrak{g} = \mathfrak{g}_{\0} \oplus \fg_{\bar 1}$ be a classical Lie superalgebra over the complex numbers $\mathbb{C}$. The representation theory of such Lie superalgebras was studied via cohomology and support variety theories in a series of early 2000s papers by Boe, Kujawa, and Nakano \cite{BKN1, BKN2, BKN3, BKN4}. They showed, among other results, that under mild assumptions on the action of the algebraic group $G_{\bar 0}$ on $\fg_{\bar 1}$, there exist subalgebras $\mathfrak{f} = \mathfrak{f}_{\0} \oplus \mathfrak{f}_{\bar 1} \leq \fg$, called detecting subalgebras, that are interesting in the sense that they have markedly simple representation theory, but they nonetheless determine the relative $(\fg, \fg_{\bar 0})$-cohomology. In particular, \cite[Theorem 4.1.1]{BKN1} gives an isomorphism $\opH^{\bullet}(\fg, \fg_{\bar 0}; \mathbb{C}) \cong \opH^{\bullet}(\mathfrak{f}, \mathfrak{f}_{\bar 0}; \mathbb{C})^{N}$, where $N$ is a non-connected reductive group determined by a choice of detecting subalgebra $\mathfrak{f}$. Moreover, $\opH^{\bullet}(\mathfrak{f}, \mathfrak{f}_{\bar 0}; \mathbb{C}) \cong \text{S}^{\bullet}(\mathfrak{f}_{\bar 1}^{*})$, and the relative cohomology for detecting subalgebras are polynomial algebras, so relative $(\fg, \fg_{\bar 0})$-cohomology is finitely generated. 

One can consider the category $\mathcal{C}_{(\fg, \fg_{\bar 0})}$ (resp. $\mathcal{F}_{(\fg, \fg_{\bar 0})}$) whose objects consist of all (resp. finite-dimensional) $\fg$-supermodules which are finitely semisimple as modules over $\fg_{\0}$ and whose morphisms consist of even morphisms between supermodules. These are abelian categories which have enough projective and injective objects. They are also Frobenius categories; i.e., projective and injective objects coincide, so one can form the stable module categories $\Stab(\mathcal{C}_{(\fg, \fg_{\bar 0})})$ and $\stab(\mathcal{F}_{(\fg, \fg_{\bar 0})})$. Objects in $\Stab(\mathcal{C}_{(\fg, \fg_{\bar 0})})$ (resp. $\stab(\mathcal{F}_{(\fg, \fg_{\bar 0}))}$) are the same as the objects in $\mathcal{C}_{(\fg, \fg_{\bar 0})}$ (resp. $\mathcal{F}_{(\fg, \fg_{\bar 0})}$), but in the stable category morphisms consist of equivalence classes of morphisms where two morphisms are considered equivalent if their difference factors through a projective module. 

\par The category $\Stab(\mathcal{C}_{(\fg, \fg_{\bar 0})})$ is a tensor triangulated category which is rigidly-compactly generated by the full tensor triangulated subcategory $\stab(\mathcal{F}_{(\fg, \fg_{\bar 0})})$ of compact-rigid objects. In \cite{BKN3} the authors consider the tensor triangular geometry of these categories. Techniques from geometric invariant theory are used to compute the categorical (Balmer) spectrum for the detecting subalgebras and for the Lie superalgebra $\mathfrak{gl}(m|n)$. For the detecting subalgebras \cite[Theorem 4.5.4]{BKN4} gives a homeomorphism $\Spc(\stab(\mathcal{F}_{(\mathfrak{f}, \mathfrak{f}_{\bar 0})})) \cong \Proj \opH^{\bullet}(\mathfrak{f}, \mathfrak{f}_{\bar 0}; \mathbb{C}) \cong \Proj \text{S}^{\bullet}(\mathfrak{f}_{\bar 1}^{*})$, but for $\fg = \mathfrak{gl}(m|n)$ the situation is a little different. The Balmer spectrum $\Spc(\stab(\mathcal{F}_{(\fg, \fg_{\bar 0})}))$ is not homeomorphic to $\Proj \opH^{\bullet}(\fg, \fg_{\bar 0}; \mathbb{C})$, as one might initially suspect. Instead, one has to consider a stack quotient of the $\Proj$ of the relative cohomology ring of a detecting subalgebra. This is the result of \cite[Theorem 5.2.2]{BKN4} which states that there is a homeomorphism $\Spc(\stab(\mathcal{F}_{(\fg, \fg_{\bar 0})})) \cong N\text{-}\Proj(\text{H}^{\bullet}(\mathfrak{f}, \mathfrak{f}_{\0}; \mathbb{C}))$. 

\subsection{}Often, the representation theory of classical Lie superalgebras over $\mathbb{C}$ resembles representation theory of finite groups in positive characteristic. If $G$ is a finite group and $k$ is an algebraically closed field of characteristic $p$ where $p$ divides the order of $G$, then the group algebra $kG$ is not semisimple and most of the time has wild representation type. The categories $kG$-Mod and $kG$-mod consisting of all (resp. finitely generated) left $kG$-modules are again abelian and Frobenius, so one can form the stable module categories $\text{StMod}(kG)$ and $\text{stmod}(kG)$ which are tensor triangulated categories. An analogous picture holds to Lie superalgebras in that $\text{StMod}(kG)$ is a compactly generated tensor triangulated category whose compact objects are precisely the objects in $\text{stmod}(kG)$. Benson, Carlson, and Rickard first studied the tensor triangular geometry of $\text{stmod}(kG)$ in the late 90s, and the main result of \cite{BCR97} is a classification of the thick $\otimes$-ideal subcategories. This in turn implies that there is a homeomorphism $\Spc(\text{stmod}(kG)) \cong \Proj \opH^{\bullet}(G, k)$. Later on in the 2000s, Benson, Iyengar, and Krause introduced the notion of a tensor triangulated category being stratified by the action of a graded-commutative ring $R$ \cite{BIK11}. This was a fundamentally new idea that not only allowed for a computation of the Balmer spectrum, but that also had the advantage of classifying $\otimes$-ideal localizing subcategories. BIK's theory was applied to show that $\text{StMod}(kG)$ is stratified by the action of the cohomology ring $\opH^{\bullet}(G, k)$, a result which recovers Benson, Carlson, and Rickard's contributions, but also gives a classification of $\otimes$-ideal localizing subcategories in terms of subsets of $\Proj \opH^{\bullet}(G, k)$. 

\subsection{}While stratification was introduced before Boe, Kujawa, and Nakano's computations of Balmer spectra for $\stab(\mathcal{F}_{(\mathfrak{f}, \mathfrak{f}_{\bar 0})})$ and $\stab(\mathcal{F}_{(\mathfrak{gl}(m|n), \mathfrak{gl}(m|n)_{\bar 0})})$, the arguments given in \cite{BKN3} do not rely on any stratification result. In fact, for most classical Lie superalgberas, e.g. $\mathfrak{gl}(m|n)$, the relative $(\fg, \fg_{\bar 0})$-cohomology ring fails to stratify. It was conjectured, however, that for a detecting subalgebra the relative cohomology ring $\text{H}^{\bullet}(\mathfrak{f}, \mathfrak{f}_{\0}; \mathbb{C})$ should stratify $\Stab(\mathcal{C}_{(\mathfrak{f}, \mathfrak{f}_{\bar 0})})$. This conjecture was proved in recent work by the author and Nakano \cite{HN24} from the observation that, just as in the case of elementary abelian two groups in characteristic two, one can reduce the problem to proving a stratification result for the stable category of modules for the exterior algebra $\Lambda^{\bullet}(\mathfrak{f}_{\bar 1})$ viewed as a superalgebra by declaring the generators to be odd, which then comes down to a version of the classical BGG correspondence. In addition to giving a classification of the $\otimes$-ideal localizing subcategories of $\Stab(\mathcal{C}_{(\mathfrak{f}, \mathfrak{f}_{\bar 0})})$, this result was also used in \cite{HN24} to prove results about representations of Lie superalgebras concerning nilpotence and the newly developed homological spectra \cite{Bal20}. Specifically it was shown that for $\mathfrak{g}$ a classical Lie superalgebra with a detecting subalgebra $\fz \leq\fg$ which is splitting in the sense of \cite{SS22}, and satisfying a natural assumption on realization of supports, then there is a homeomorphism $\Spc^{\text{h}}(\stab(\mathcal{F}_{(\fg, \fg_{\bar 0})})) \cong N\text{-}\Proj(\text{H}^{\bullet}(\fz, \fz_{\0}; \mathbb{C}))$. This result led to a verification of Balmer's ``Nerves-of-Steel" conjecture for $\fg = \mathfrak{gl}(m|n)$ \cite[Theorem 7.2.1]{HN24}. 

\subsection{} While the BIK approach to stratifying tensor triangulated categories represented a major breakthrough in the field, the drawback is that it requires that the category be equipped with an action by a graded-commutative Noetherian ring in order to construct the necessary support theory. The case of $\stab(\mathcal{F}_{(\fg, \fg_{\bar 0})})$ is not unique in that in fact, for many important tensor triangulated categories a stratifying ring is not available, rendering them outside of the scope of the BIK machinery. This prompted Barthel, Heard, and Sanders in \cite{BHS21} to develop a theory of stratification called tensor triangular stratification or tt-stratification for short. Tt-stratification for compactly generated tensor triangulated categories is in terms of the support theory developed by Balmer and Favi \cite{BF11} for ``big" objects, and has the advantage that does not presuppose the action of a graded-commutative Noetherian ring. Instead, the supports are constructed using the more intrinsic Balmer spectrum of the full compact subcategory. A drawback to tt-stratification is that descent theorems still have to be handled one TTC at a time. Recently, in \cite{BHSZ24}, Barthel, Heard, Sanders, and Zou developed a third theory of stratification, called h-stratification, which is in terms of the homological spectrum. It is also well behaved under descent. These key insights allow us, in particular, to reduce the problem of studying $\glmn$ to studying a splitting, detecting subalgebra, as well as subalgebras isomorphic to the Lie superalgebra $\mathfrak{q}(1)$. 

The purpose of this paper is to expand upon the work of \cite{HN24} by adding categories of representations of Lie superalgebras to the growing list of tensor triangulated categories for which stratification results are known. We investigate the three notions of stratification for $\Stab(\mathcal{C}_{(\fg, \fg_{\bar 0})})$ when $\mathfrak{g}$ is a classical Lie superalgebra which has a splitting detecting subalgebra $\fz \leq\fg$. Specifically, our theorem is the following.


\begin{mainthm}\label{mainthm}
Let $\fg$ be a classical Lie superalgebra with a splitting detecting subalgebra $\fz \leq\fg$ and which satisfies the realization condition of Assumption \ref{assumption} \cite[Theorem 7.2.1]{HN24}. The tensor triangulated category $\Stab(\mathcal{C}_{(\fg, \fg_{\bar 0})})$ is tt-stratified by the Balmer spectrum $\Spc(\stab(\mathcal{F}_{(\fg, \fg_{\bar 0})}))$, and tt-stratification is equivalent to h-stratification. 
\end{mainthm}


Classical Lie superalgebras which satisfy the hypotheses of Theorem A include classical Lie superalgebras of Type A. As a consequence of stratification, we obtain the classification of $\otimes$-ideal localizing subcategories in terms of arbitrary subsets of the Balmer spectrum. This fact, combined with the Balmer spectrum computations of Boe, Kujawa, and Nakano, yields the following corollary. 


\begin{maincor}\label{maincorollary}
Let $\fg$ be a classical Lie superalgebra with a splitting detecting subalgebra $\fz \leq\fg$ and which satisfies the realization condition of Assumption \ref{assumption}. There is a bijection between the set of $\otimes$-ideal localizing subcategories of $\Stab(\mathcal{C}_{(\fg, \fg_{\bar 0})})$ and subsets of $N\text{-}\Proj(\text{H}^{\bullet}(\fz, \fz_{\0}; \mathbb{C}))$.
\end{maincor}


Our work is inspired by the strategy outlined for modular representation theory in \cite{BIK} where BIK stratification was proven for $\text{stmod}(kG)$ where $G$ is an arbitrary group by first proving the result for elementary abelian $G$ is an elementary abelians, and then ``bootstrapping" up by applying a version of Quillen's stratification theorem and Choinard's theorem. Their argument involves several ``change-of-categories" type arguments including a consideration of the homotopy category of injectives $\mathbf{K}(\text{Inj }kG)$. We make use of some of the analogous results in our setting, specifically for studying detecting subalgebras.


\subsection*{Acknowledgments}      
The author would like to express his gratitude to his Ph.D. advisor, Daniel K. Nakano, for his guidance and encouragement throughout this project.


\section{Tensor triangulated categories}

\subsection{Triangulated Categories}

Recall that a triangulated category $\mathscr{T}$ is an additive category equipped with an auto-equivalence $\Sigma : \mathscr{T} \to \mathscr{T}$ called the shift, and a class of distinguished triangles: 
$$
M \to N \to Q \to \Sigma M
$$
all subject to a list of axioms the reader can find in, for example, \cite[Ch. 1]{Nee01}. 
\par A non-empty, full, additive subcategory $\mathscr{S}$ of $\mathscr{T}$ is called a \emph{triangulated subcategory} if (i) $M \in \mathscr{S}$ implies that $\Sigma^n M \in \mathscr{S}$ for all $n \in \mathbb{Z}$ and (ii) if $M \to N \to Q \to \Sigma M$ is a distinguished triangle in $\mathscr{T},$ and two of $\{M, N, Q\}$ are objects in $\mathscr{S},$ then the third object is also in $\mathscr{S}.$ A triangulated subcategory $\mathscr{S}$ of $\mathscr{T}$ is called \emph{thick} if $\mathscr{S}$ is closed under taking direct summands. 

A triangulated subcategory $\mathscr{S}$ of $\mathscr{T}$ is called a \emph{localizing subcategory} if $\mathscr{S}$ is closed under taking set-indexed coproducts. It follows from a version of the Eilenberg swindle that localizing subcategories are necessarily thick. It is a central problem in the subject to, given a triangulated category, classify its thick subcategories and localizing subcategories. 

An object $C \in \mathscr{T}$ is called \emph{compact} if $\Hom_{\mathscr{T}}(C,-)$ commutes with set-indexed coproducts. The full subcategory of compact objects in $\mathscr{T}$ is denoted by $\mathscr{T}^c$, and the triangulated category $\mathscr{T}$ is said to be \emph{compactly generated} if the isomorphism classes of compact objects form a set, and if for each non-zero object $M \in \mathscr{T}$ there is an object $C \in \mathscr{T}^c$ such that $\Hom_{\mathscr{T}}(C,M) \neq 0$.

\subsection{Tensor triangulated categories}
A tensor triangulated category (TTC) is a triple $(\mathscr{K}, \otimes, \unit)$ consisting of a triangulated category $\mathscr{K}$, a symmetric, monoidal (tensor) product $\otimes: \mathscr{K} \times \mathscr{K} \to \mathscr{K}$ which is exact in each variable, and a monoidal unit $\unit$. For the remainder of this section, $\mathscr{K}$ denotes a TTC.

The usual paradigm involves the situation when the TTC $\mathscr{K}$ is rigidly-compactly generated as a TTC by its full subcategory of compact-rigid objects $\mathscr{K}^c$. By definition this means that (i) $\mathscr{K}$ is closed under set indexed coproducts, (ii) the tensor product preserves set-indexed coproducts, (iii) $\mathscr{K}$ is compactly generated as a triangulated category, (iv) the tensor product of compact objects is compact, (v) $\unit$ is a compact object, and (vi) every compact object is rigid (i.e. dualizable). 

The additional structure of the tensor product possessed by TTCs makes the problem of classifying localizing subcategories and thick ideals more tractable. When working in the context of a TTC however, one often focuses on classifications for $\otimes$-ideal localizing subcategories and thick $\otimes$-ideal subcategories of compact objects. In many situations these match up with the purely triangular notions (c.f. \cite[Section 4.1.4]{BIK}). For example, if $\mathscr{K}$ is monogenic (generated by the unit object $\unit$), as in the case of finite groups, then every localizing subcategory is automatically a $\otimes$-ideal. One difference between previously considered representation categories and our theory is that the categories of Lie superalgebra representations of interest in this paper are in general \emph{not} monogenic. Because of this feature our available tools restrict us to a classification of $\otimes$-ideal localizing subcategories rather than arbitrary localizing subcategories.

\subsection{Support for objects in TTCs}
The study of TTCs frequently involves considering various support spaces and notions of support for objects. The relevant definition here is that of a support datum, which was originally given in \cite{Bal05}. The definition we give is slightly more general, and is more suited to our paper. Let $\mathscr{K}$ be a TTC, $X$ be a Zariski topological space (\cite[Section 2.3]{BKN3}), and let $\mathcal{X}$ denote the collection of all subsets of $X$. A \emph{support datum} on $\mathscr{K}$ is an assignment $V: \mathscr{K} \to \mathcal{X}$ such that the following properties hold for $M, N, M_i, Q$ objects in $\mathscr{K}$: 
\begin{itemize}

\item[(2.3.1)] $V(0) = \emptyset$, and $V(\unit) = X$; 
\item[(2.3.2)] $V(\oplus_{i \in I}M_i) = \cup_{i \in I}V(M_i)$ provided that $\oplus_{i \in I}M_i$ is an object of $\mathscr{K}$; 
\item[(2.3.3)] $V(\Sigma M) = V(M)$
\item[(2.3.4)] for any distinguished triangle $M \to N \to Q \to \Sigma M$, $V(N) \subseteq V(M) \cup V(Q)$; 
\item[(2.3.5)] $V(M \otimes N) = V(M) \cap V(N)$; 
\end{itemize}
Often useful are support data which satisfy the additional two properties: 
\begin{itemize}
\item[(2.3.6)] $V(M) = \emptyset$ if and only if $M = 0$. 
\item[(2.3.7)] for any closed subset $W \in \mathcal{X}$, there exists an object $M$ in $\mathscr{K}^c$ such that $V(M) = W$. 
\end{itemize}
Property (2.3.7) is often called the \emph{realization property}, and support data which satisfy Property 2.3.6 are called \emph{faithful}. 

\subsection{The Balmer spectrum of a TTC} In his foundational 2005 paper \cite{Bal05}, Balmer uses the tensor product and unit object to associate to each TTC, $\mathscr{K}$, a topological space known as the categorical (Balmer) spectrum $\Spc(\mathscr{K})$ in a way analogous to the construction of the prime spectrum of a commutative ring. Define a $\otimes$-ideal in $\mathscr{K}$ as a full triangulated subcategory $\mathscr{I}$ of $\mathscr{K}$ such that $M \otimes N \in \mathscr{I}$ for all $M \in \mathscr{I}$ and $N \in \mathscr{K}$. A proper, thick $\otimes$-ideal $\mathscr{P}$ of $\mathscr{K}$ is said to be prime if, for objects $M$ and $N$ in $\mathscr{K}$, $M \otimes N \in \mathscr{P}$ implies that $M \in \mathscr{P}$ or $N \in \mathscr{P}$. The Balmer spectrum is then defined as 
$$
\Spc(\mathscr{K}) := \{\mathscr{P} \subsetneq \mathscr{K} \ | \ \mathscr{P} \text{ is a prime ideal}\}.
$$
The topology on $\Spc(\mathscr{K})$ is the familiar Zariski topology which has the closed sets given by 
$$
Z(\mathscr{C}) := \{\mathscr{P} \in \Spc(\mathscr{K}) \ | \ \mathscr{C} \cap \mathscr{P} \neq \emptyset\},
$$
where $\mathscr{C}$ is an arbitrary collection of objects in $\mathscr{K}$.


An important result of Balmer, \cite[Theorem 3.2]{Bal05}, has to do with supports for essentially small TTCs constructed via the categorical spectrum. One can construct a support datum on $\mathscr{K}^c$ as follows. Given an object $M \in \mathscr{K}^c$, define 
$$
\supp_{\text{Bal}}(M) := \{\mathscr{P} \in \Spc(\mathscr{K}^c) \ | \ M \notin \mathscr{P}\}.
$$
Balmer showed that the support datum given by $(\Spc(\mathscr{K}^c), \supp_{\text{Bal}}(-))$ is universal in the sense that if $(X, V)$ is any support datum on $\mathscr{K}^c$, then there exists a unique continuous map $f: X \to \Spc(\mathscr{K}^c)$ such that $V(M) = f^{-1}(\supp_{\text{Bal}}(M))$. 

\subsection{Extending support to big objects} Let $\mathscr{K}$ be a rigidly-compactly generated TTC, and let $\mathscr{K}^c$ denote its full subcategory of compact-rigid objects. One might try to construct a universal notion of support for big TTCs which generalizes Balmer's construction for essentially small situations. Morally speaking, Balmer's construction fails to provide such a universal support for big objects because one expects that supports for big objects should be open, while Balmer supports are closed by definition. It turns out that no one has succeeded in constructing universal support for big objects, and it may be impossible to construct such a support datum in general \cite{BKS20}. Nonetheless, we are still interested in constructing supports for big objects which, though not universal, still prove useful in practice. The following definition is then motivated.

\begin{definition}
Let $\mathscr{K}$ be a rigidly-compactly generated TTC and let $\mathscr{K}^c$ denote the full subcategory of compact-rigid objects. Let $(X, V)$ be a support datum on $\mathscr{K}^c$ where supports are closed subsets of $X$. The pair $(X, \mathcal{V})$ is said to be an \emph{extension} of $(X, V)$ if $\mathcal{V}$ is an assignment from $\mathscr{K}$ to arbitrary subsets of $X$ satisfying the following: 

\begin{itemize}
\item[(a)] $\mathcal{V}$ satisfies properties (2.3.1)-(2.3.5) for objects in $\mathscr{K}$; 
\item[(b)] $\mathcal{V}(M) = V(M)$ for all $M$ in $\mathscr{K}^c$; and 
\item[(c)] if $V$ satisfies (2.3.7) then $\mathcal{V}$ satisfies (2.3.7). 
\end{itemize}
\end{definition}

A notable example of such an extension of support data are the Balmer-Favi supports which are reviewed in Section 2.6. See also \cite{Bal20}. 


\subsection{BIK support}
Let $\mathscr{K}$ be a rigidly-compactly generated TTC, and let $\mathscr{K}^c$ denote the full subcategory of compact-rigid objects. Aimed at obtaining tensor triangular classifications, Benson, Iyengar, and Krause defined an extension of Balmer's universal support datum for big objects. The additional assumption needed throughout this subsection necessary to define these so-called BIK supports is that of an auxiliary canonical action of a graded-commutative ring $R$ on $\mathscr{K}$. Let us explain what this means. One can always consider the graded-center $Z^{\bullet}(\mathscr{K})$ of $\mathscr{K}$ which is a graded-commutative ring whose degree $n$ component is given by 
$$
Z^n(\mathscr{K}) = \{ \eta : \text{Id}_{\mathscr{K}} \longrightarrow \Sigma^n \ | \ \eta\Sigma = (-1)^n \Sigma\eta \}.
$$
An action of $R$ on $\mathscr{K}$ is a homomorphism of graded-commutative rings $\phi : R \to Z^{\bullet}(\mathscr{K})$. If $\mathscr{K}$ admits an $R$-action, then $\mathscr{K}$ is called $R$-linear. 
\par Given objects $M$ and $N$ in $\mathscr{K},$ set  
$$
\Hom^{\bullet}_{\mathscr{K}}(M,N) := \bigoplus_{i \in \mathbb{Z}}\Hom_{\mathscr{K}}(M, \Sigma^i N). 
$$
Then $\Hom^{\bullet}_{\mathscr{K}}(M,N)$ is a graded abelian group, and $\End^{\bullet}_{\mathscr{K}}(M):= \Hom^{\bullet}_{\mathscr{K}}(M, M)$ is a graded ring where the multiplication is given by applying the shift and then composing morphisms. Notice that $\Hom^{\bullet}_{\mathscr{K}}(M,N)$ is a right $\End^{\bullet}_{\mathscr{K}}(M)$ and a left $\End^{\bullet}_{\mathscr{K}}(N)$-bimodule. It follows that $\mathscr{K}$ being $R$-linear is equivalent to their being, for each object $M$ in $\mathscr{K}$, an induced homomorphism of graded rings $\phi_M : R \to \End^{\bullet}_{\mathscr{K}}(M)$ such that the induced $R$-module structures on $\Hom^{\bullet}_{\mathscr{K}}(M,N)$ by $\phi_M$ and $\phi_N$ agree up to the usual sign.
\par Using the tensor product in $\mathscr{K}$ allows one to construct an action of the graded endomorphism ring $\End^{\bullet}_{\mathscr{K}}(\unit)$ of the unit object via the maps defined by taking, for each $M$ in $\mathscr{K}$, $\phi_M : \End^{\bullet}_{\mathscr{K}}(\unit) \to \End^{\bullet}_{\mathscr{K}}(M)$ given by tensoring with $M$. Provided that $\End^{\bullet}_{\mathscr{K}}(\unit)$ is Noetherian, any action on $\mathscr{K}$ induced via this action is called canonical. 
\par BIK supports for objects in $\mathscr{K}$ are given in terms of the homogenous prime ideal spectrum $\Proj R$ of $R$. Defining support in terms of $\Proj R$ has its origins in the work of Benson, Carlson, and Rickard \cite{BCR97}, who use this support to classify thick subcategories of the stable module category of a finite group. For each $\mathfrak{p} \in \Proj R$, a deep result in Bousfield localization allows one to construct an exact, local cohomology functor $\mathit{\Gamma}_{\mathfrak{p}}: \mathscr{K} \to \mathscr{K}$. Properties of these local cohomology functors can be found in Section 3.1.2 of \cite{BIK}. The space of BIK supports for $\mathscr{K}$ is defined to be 
$$
\Supp_{\text{BIK}}(\mathscr{K}) := \{ \mathfrak{p} \in \Proj(R) \ | \ \mathit{\Gamma_{\mathfrak{p}}}(\mathscr{K}) \neq 0 \}. 
$$
For an object $M \in \mathscr{K}$, the BIK support of $M$ is defined as 
$$
\Supp_{\text{BIK}}(M) := \{ \mathfrak{p} \in \Proj(R) \ | \ \mathit{\Gamma_{\mathfrak{p}}}(M) \neq 0 \}. 
$$
That $(\Supp_{\text{BIK}}(\mathscr{K}),\Supp_{\text{BIK}}(-))$ is a support data on $\mathscr{K}$ which extends $(\Spc(\mathscr{K}^c), \supp_{\text{Bal}}(-))$ is the main content \cite[Chapter 3]{BIK}.

\subsection{BIK stratification}
Let $\mathscr{K}$ be a rigidly-compactly generated TTC, and $\mathscr{K}^c$ denote the full subcategory of compact objects of $\mathscr{K}$. Let $R$ be a graded-commutative Noetherian ring, and assume that $\mathscr{K}$ is $R$-linear. For this section only we follow the lead of \cite{BIK} by assuming furthermore that $\mathscr{K}$ is \emph{monogenic}; i.e., that $\mathscr{K}$ is compactly generated by the unit object $\unit$ of $\mathscr{K}$.
\par As a first application of BIK supports, one can construct maps between the collection of $\otimes$-ideal localizing subcategories of $\mathscr{K}$ and arbitrary subsets of $\Proj R$. The maps are defined as follows. Given a $\otimes$-ideal localizing subcategory $\mathscr{C}$ of $\mathscr{K}$, set 
$$
\sigma(\mathscr{C}) = \Supp_{\text{BIK}}(\mathscr{C}) = \{\mathfrak{p} \in \Proj(R) \ | \ \mathit{\Gamma_{\mathfrak{p}}}(\mathscr{C}) \neq 0\}. 
$$
Next, given a subset $V$ of $\Proj(R)$, set 
$$
\tau(V) = \{M \in \mathscr{K} \ | \ \Supp_{\text{BIK}}(M) \subseteq V\}.
$$ 
BIK stratification has to do with two conditions which, when satisfied, guarantee that $\sigma$ and $\tau$ provide mutually inverse bijections. The two conditions are the following. 

\begin{itemize} 

\item[(2.7.1)] \emph{The BIK local-to-global principle}: for each object $M$ in $\mathscr{K}$, 
\begin{center} 
$\Loc_{\otimes}\langle \{M \} \rangle = \Loc_{\otimes}\langle \{\mathit{\Gamma}_{\mathfrak{p}}M \ | \ \mathfrak{p} \in \Proj R\} \rangle$. 
\end{center}

\item[(2.7.2)] \emph{The BIK minimality condition}: for $\mathfrak{p} \in \Supp_{\text{BIK}}(\mathscr{K})$, the subcategory $\mathit{\Gamma}_{\mathfrak{p}}\mathscr{K}$ is a minimal $\otimes$-ideal localizing subcategory of $\mathscr{K}$. 

\end{itemize}

When conditions (2.7.1) and (2.7.2) hold, $\mathscr{K}$ is said to be \emph{stratified in the sense of BIK}. The following theorem \cite[Theorem 4.19]{BIK} gives the classification. 

\begin{theorem}
If $\mathscr{K}$ is stratified in the sense of BIK, then the maps $\sigma$ and $\tau$ provide mutually inverse bijections between the set of $\otimes$-ideal localizing subcategories of $\mathscr{K}$ and subsets of $\Proj R$: 
$$
\{\text{$\otimes$-ideal localizing subcategories of } \mathscr{K}\} \xleftrightarrow[\tau]{\sigma} \{\text{subsets of } \Supp_{\text{BIK}} \mathscr{K}\}.
$$

\end{theorem}

A useful fact that applies in many practical situations is that the BIK local-to-global principle automatically holds in instances where the Krull dimension of $\Proj R$ is finite. This reduces much of the work involved with verifying BIK stratification to the minimality condition. 


\subsection{Balmer-Favi support}
Throughout this section, let $\mathscr{K}$ be a rigidly-compactly generated TTC, and let $\mathscr{K}^c$ denote the full subcategory of compact-rigid objects. In \cite{BF11}, Balmer and Favi construct a support datum on $\mathscr{K}$ which extends Balmer's universal support for $\mathscr{K}^c$. The key difference, from our point of view, between BIK supports and Balmer-Favi supports is that constructing Balmer-Favi supports does not make use of an auxiliary ring action. Instead, supports are constructed based on certain tensor idempotents which themselves are a tensor triangular abstraction of Rickard's idempotent modules from \cite{Ric97}. Another difference is that the space of supports is always the Balmer spectrum $\Spc(\mathscr{K}^c)$ of the compact subcategory, which is not necessarily the case for BIK supports. Throughout this section only we impose the additional hypothesis that $\Spc(\mathscr{K}^c)$ is Noetherian. One can actually get away with a weaker condition, namely assuming that $\Spc(\mathscr{K}^c)$ is weakly Noetherian (cf. \cite[Section 1]{BHS21}). However, since all Balmer spectra we are concerned with are Noetherian, working under a Noetherian hypothesis simplifies the exposition. The driving force behind the scenes is again Bousfield localization, which guarantees that for every specialization closed subset $W \subset \Spc(\mathscr{K}^c)$, one can construct two $\otimes$-idempotents $E(W) \cong E(W) \otimes E(W)$ and $F(W) \cong F(W) \otimes F(W)$ in $\mathscr{K}$ that fit into a distinguished triangle 
$$
E(W) \to \unit \to F(W) \to \Sigma E(W). 
$$ 
The Noetherian hypothesis on $\Spc(\mathscr{K}^c)$ implies that every point $\mathscr{P} \in \Spc(\mathscr{K}^c)$ is visible in the sense of \cite[Section 7.9]{BF11}. From this, \cite[Lemma 7.8]{BF11} gives that one can express each $\mathscr{P} \in \Spc(\mathscr{K}^c)$ as $\mathscr{P} = Y_1 \cap Y_2^c$ for specialization closed subsets $Y_1, Y_2 \subseteq \Spc(\mathscr{K}^c)$. One then defines a $\otimes$-idempotent $g(\mathscr{P}) := E(Y_1) \otimes F(Y_2)$ that depends only on $\mathscr{P}$ and not the choice of specialization closed subsets. For an object $M \in \mathscr{K}$, the Balmer-Favi support of $M$ is defined to be 
$$
\Supp_{\text{BF}}(M) := \{\mathscr{P} \in \Spc(\mathscr{K}^c) \ | \ M \otimes g(\mathscr{P}) \neq 0\}. 
$$
That $(\Supp_{\text{BF}}(-), \Spc(\mathscr{K}^c))$ is a support datum on $\mathscr{K}$ that extends $(\Spc(\mathscr{K}^c), \supp_{\text{Bal}}(-))$ is the content of \cite[Prop. 7.18]{BF11}.


\subsection{tt-stratification via Balmer-Favi supports}
Balmer and Favi's original motivation that led to the definition of Balmer-Favi support was to transport methods from modular representation theory to algebraic geometry. A latent application however, realized by Barthel, Heard, and Sanders in \cite{BHS21}, is a stratification theory distinct from the theory developed by BIK which frees one from the burdensome requirement of a ring action, and which is universal in Noetherian situations. Again, the stratification framework begins with defining maps that allow one to pass between $\otimes$-ideal localizing subcategories of $\mathscr{K}$ and arbitrary subsets of $\Spc(\mathscr{K}^c)$. To that end, given a $\otimes$-ideal localizing subcategory $\mathscr{C}$ of $\mathscr{K}$, set 
$$
\sigma(\mathscr{C}) = \bigcup_{M \in \mathscr{C}} \Supp_{\text{BF}}(M). 
$$
Then, given a subset $V$ of $\Spc(\mathscr{K}^c)$, set 
$$
\tau(V) = \{M \in \mathscr{K} \ | \ \Supp_{\text{BF}}(M) \subseteq V\}. 
$$

The relevant conditions on Balmer-Favi supports which allow for tensor triangular classifications are the following. 

\begin{itemize} 

\item[(2.9.1)] \emph{The tt local-to-global principle}: for each object $M$ in $\mathscr{K}$, 
\begin{center} 
$\Loc_{\otimes}\langle \{M \} \rangle = \Loc_{\otimes}\langle \{M \otimes g(\mathscr{P}) \ | \ \mathscr{P} \in \Spc(\mathscr{K}^c)\} \rangle$. 
\end{center}

\item[(2.9.2)] \emph{The tt minimality condition}: for each $\mathscr{P} \in \Spc(\mathscr{K}^c)$, the subcategory $\Loc_{\otimes}\langle \{g(\mathscr{P})\} \rangle$ is a minimal $\otimes$-ideal localizing subcategory of $\mathscr{K}$. 

\end{itemize}

When conditions (2.9.1) and (2.9.2) hold, $\mathscr{K}$ is said to be \emph{tt-stratified}, and \cite[Theorem 4.1]{BHS21} gives the classification. We record a condensed version here for completeness. 
\begin{theorem}
If $\mathscr{K}$ is tt-stratified, then the maps $\sigma$ and $\tau$ provide bijections between the set of $\otimes$-ideal localizing subcategories of $\mathscr{K}$ and subsets of $\Spc(\mathscr{K}^c)$: 
$$
\{\text{$\otimes$-ideal localizing subcategories of } \mathscr{K}\} \xleftrightarrow[\tau]{\sigma} \{\text{subsets of } \Supp_{\text{BF}} \mathscr{K}\}.
$$

\end{theorem}

\subsection{Homological support}
The setup is again the same. Namely, we work in the context where $\mathscr{K}$ is a rigidly-compactly generated TTC, and the full subcategory of compact-rigid objects is denoted $\mathscr{K}^c$. Motivated by a desire for abstract nilpotence theorems in TTCs, Balmer defined a topological space called homological spectrum $\Spc^{\text{h}}(\mathscr{K}^c)$ (cf. \cite{Bal20}). A condensed review of the construction is as follows. Let $\mathbf{Ab}$ denote the category of abelian groups. The category $\Mod$-$\mathscr{K}^c$ of right $\mathscr{K}^c$-modules is the category whose objects consist of additive functors $M: (\mathscr{K}^c)^{\text{op}} \to \mathbf{Ab}$, and whose morphisms consist of natural transformations between functors. The module category $\Mod$-$\mathscr{K}^c$ is an abelian category and receives $\mathscr{K}^c$ via the Yoneda embedding which we denote by 
\begin{align*}
\text{h} : \ \mathscr{K}^c &\hookrightarrow \Mod\text{-}\mathscr{K}^c =  \Add((\mathscr{K}^c)^{\text{op}}, \mathbf{Ab})\\
M &\mapsto \hat{M} := \Hom_{\mathscr{K}^c}(-,M) \\
f &\mapsto \hat{f}.
\end{align*}
Let $\mathscr{A}$ denote $\Mod$-$\mathscr{K}^c$. Day convolution gives $\mathscr{A}$ a tensor structure which is colimit-preserving in each variable and which makes $\text{h}$ a tensor functor. The tensor subcategory $\mathscr{A}^{\text{fp}}:=\text{mod-}\mathscr{K}^c \subseteq \mathscr{A}$ of finitely presented objects is the Freyd envelope of $\mathscr{K}^c$. A homological prime of $\mathscr{K}^c$ is defined to be a maximal, proper, Serre $\otimes$-ideal subcategory $\mathscr{B} \subseteq \mathscr{A}^{\text{fp}}$, and the homological spectrum of $\mathscr{K}^c$ is defined to be the set of homological primes: 
$$
\Spc^{\text{h}}(\mathscr{K}^c) := \{\mathscr{B} \subseteq \mathscr{A}^{\text{fp}} \ | \ \mathscr{B} \text{ is a maximal Serre tensor ideal subcategory}\}. 
$$
One can define a support datum on $\mathscr{K}^c$ in terms of the homological spectrum by defining the homological support of on object $M$ in $\mathscr{K}^c$ as 
$$
\supp^{\text{h}}(M) := \{\mathscr{B} \in \Spc^{\text{h}}(\mathscr{K}^c) \ | \ \hat{M} \notin \mathscr{B}\}. 
$$
One can view the collection $\supp^{\text{h}}(M)$ as $M$ ranges over all objects of $\mathscr{K}^c$ as a basis for the closed subsets of a topology on $\Spc^{\text{h}}(\mathscr{K}^c)$. The universality of the Balmer spectrum and Balmer supports for objects of $\mathscr{K}^c$ implies the existence of a unique continuous map $\phi: \Spc^{\text{h}}(\mathscr{K}^c) \to \Spc(\mathscr{K}^c)$ called the comparison map. It is known to be surjective assuming the rigidity of $\mathscr{K}^c$. In all known examples, the comparison map is a bijection. This leads to the following. 
\begin{conjecture}[Nerves-of-Steel]
Let $\mathscr{K}^c$ be rigid. The comparison map 
$$
\phi: \Spc^{\text{h}}(\mathscr{K}^c) \to \Spc(\mathscr{K}^c) 
$$
is a bijection. 
\end{conjecture}
The homological support for objects in $\mathscr{K}^c$ can be extended to a support datum on $\mathscr{K}$ via the following construction. From \cite[Section 3]{BKS19} there is a pure-injective object $E_{\mathscr{B}}$ in $\mathscr{K}$ corresponding to each homological prime $\mathscr{B} \in \Spc^{\text{h}}(\mathscr{K}^c)$. For objects $M$ in $\mathscr{K}$, the extended homological support is defined as
$$
\Supp^{\text{h}}(M) := \{\mathscr{B} \in \Spc^{\text{h}}(\mathscr{K}^c) \ | \ \textsf{hom}(M, E_{\mathscr{B}}) \neq 0\}, 
$$
where $\textsf{hom}(-,-)$ denotes the internal hom in $\mathscr{K}$.

\subsection{Stratification via homological support}
In recent work, Barthel, Heard, Sanders, and Zou \cite{BHSZ24} developed a notion of stratification in terms of the homological spectrum and homological support. This theory of stratification, called homological stratification or h-stratification, has the advantage of satisfying a very general form of descent. Let $\mathscr{K}$ be a rigidly-compactly generated TTC, and let $\mathscr{K}^c$ denote the full subcategory of compact-rigid objects. Homological support determines natural maps $\sigma$ and $\tau$ between $\otimes$-ideal localizing subcategories of $\mathscr{K}$ and subsets of $\Spc^{\text{h}}(\mathscr{K}^c)$. The maps are constructed in a similar way to how they are constructed for tt-stratification. Given a $\otimes$-ideal localizing subcategory $\mathscr{C}$ of $\mathscr{K}$, set 
$$
\sigma(\mathscr{C}) = \bigcup_{M \in \mathscr{C}} \Supp^{\text{h}}(M). 
$$
Then, given a subset $V$ of $\Spc(\mathscr{K}^c)$, set 
$$
\tau(V) = \{M \in \mathscr{K} \ | \ \Supp^{\text{h}}(M) \subseteq V\}. 
$$
As usual $\mathscr{K}$ is said to be homologically stratified if these maps give mutually inverse bijections. The theorem that gives sufficient and necessary conditions for homological stratification is given in \cite{BHSZ24}. We record a slightly modified form of the theorem here. 
\begin{theorem}
The TTC $\mathscr{K}$ is homologically stratified if the following conditions hold. 

\begin{itemize} 

\item[(2.11.1)] \emph{The homological local-to-global principle}: for each object $M$ in $\mathscr{K}$, 
\begin{center} 
$\Loc_{\otimes}\langle \{M \} \rangle = \Loc_{\otimes}\langle \{M \otimes E_{\mathscr{B}}) \ | \ \mathscr{B} \in \Spc^{\text{h}}(\mathscr{K}^c)\} \rangle$. 
\end{center}

\item[(2.11.2)] \emph{The homological minimality condition}: for each $\mathscr{B} \in \Spc^{\text{h}}(\mathscr{K}^c)$, the subcategory $\Loc_{\otimes}\langle \{E_{\mathscr{B}}\} \rangle$ is a minimal $\otimes$-ideal localizing subcategory of $\mathscr{K}$. 

\end{itemize}

\end{theorem}

\subsection{Relationships between stratification theories, pros and cons, etc.}
A developing theme in the area involves the comparison between the various notions of stratification and the implications that allow one to pass from one notion of stratification to another. As we will see sometimes it is relatively straightforward to obtain results for a TTC for a particular stratification theory, but for one reason or another it is convenient to be able to transfer the results to the other notions where other results can be applied to say something new. First let us consider the pros and cons. 

\begin{center}
\begin{tabular}{ | c | c | c | }
\hline
Type of stratification & Pros & Cons \\ 
\hline
\multirow{2}{4em}{BIK} & $\bullet$ Neeman's theorem to work with & $\bullet$ Requires ring action \\
& $\bullet$ Many results in the literature available & $\bullet$ $\End^{\bullet}(\unit)$ may not help\\ 
\hline 
\multirow{2}{4em}{Tensor triangular} & $\bullet$ In terms of $\Spc(\mathscr{K}^c)$& $\bullet$ Need to understand $\Spc(\mathscr{K}^c)$ \\
& $\bullet$ No ring action required & $\bullet$ No general form of descent  \\ 
\hline
\multirow{2}{4em}{Homological} & $\bullet$ $\Spc^{\text{h}}(\mathscr{K}^c)$ often concrete& $\bullet$ Often need Nerves-of-Steel \\
& $\bullet$ Very general form of descent & $\bullet$ Local-global principle not trivial  \\ 
\hline

\end{tabular}

\end{center}

The first implication result we want to highlight \cite[Theorem D]{BHS21} which demonstrates the universality of tt-stratification in Noetherian situations. This can be interpreted as ``BIK stratification implies tt-stratification". 
\begin{theorem}
Let $\mathscr{K}$ be a rigidly-compactly generated TTC which is Noetherian and stratified in the sense of BIK by the action of a graded-commutative Noetherian ring $R$. Then the BIK space of supports $\Supp_{\text{BIK}}(\mathscr{K})$ is canonically homeomorphic to $\Spc(\mathscr{K}^c)$ and the BIK notion of support coincides with the Balmer-Favi notion of support. 
\end{theorem}
\par The relationships between tt-stratification and h-stratification are not as clear-cut. This has to do with the fact that h-stratification does not rely on any point-set topological conditions on the homological spectrum, so h-stratification does not have a strong universality statement available. However, in weakly Noetherian situations, and when the Nerves-of-Steel Conjecture holds, the relationship is tight, as given in \cite[Theorem E]{BHSZ24}, which we record here. 
\begin{theorem}\label{nos}
If $\mathscr{K}$ is a rigidly-compactly generated TTC with $\Spc(\mathscr{K}^c)$ weakly Noetherian, then the following are equivalent: 
\begin{itemize}
\item[(a)] $\mathscr{K}$ is tt-stratified; 
\item[(b)] $\mathscr{K}$ is h-stratified and the Nerves-of-Steel Conjecture holds for $\mathscr{K}$. 
\end{itemize}
\end{theorem}
As the original authors note, this means that when the Nerves-of-Steel Conjecture holds, it suffices to consider only h-stratification. 
\par Finally, we want to mention descent. Before h-stratification, descent results had to be verified one TTC at a time, and general methods were not available. The following theorem \cite[Theorem A]{BHSZ24} demonstrates the power of h-stratification in contexts where natural restriction functors exist. 
\begin{theorem}\label{descent}
Let $(f_i^* : \mathscr{K} \to \mathscr{S}_i)_{i \in I}$ be a family of exact, symmetric monoidal functors that preserve set-indexed coproducts and jointly detect when an object of $\mathscr{K}$ is zero. Suppose that $\mathscr{S}_i$ is tt-stratified for all $i \in I$. If $\mathscr{K}$ satisfies the Nerves-of-Steel Conjecture and has a weakly Noetherian spectrum, then the following are equivalent: 
\begin{itemize}
\item[(a)] $\mathscr{K}$ is tt-stratified
\item[(b)] $\mathscr{K}$ is h-stratified 
\item[(c)] $\mathscr{K}$ is generated by the images of the right adjoints $(f_i)_*$
\end{itemize}
\end{theorem}
This generalizes all the descent results in the literature, and provides a uniform approach for developing new results, as we will lay out for Lie superalgebras in the following sections. 



\section{Lie superalgebras}

We give an overview of basic concepts from the theory of Lie superalgebras and their representations. Most of this material is covered in \cite{Kac1} and \cite{BKN1}. Fix $\mathbb{C}$, the complex numbers, as the ground field in all that follows. Unless stated otherwise, $\otimes$ means $\otimes_{\mathbb{C}}$.

\subsection{Lie superalgebras and their representations} A Lie superalgebra is a super vector space $\fg = \fg_{\0} \oplus \fg_{\bar 1}$ equipped with a bilinear bracket $[-,-] : \fg \otimes \fg \to \fg$ which satisfies $\mathbb{Z}/2\mathbb{Z}$-graded versions of the usual skew-symmetry and Jacobi axioms for Lie algebras. Let $\fg = \fg_{\0} \oplus \fg_{\bar 1}$ be a Lie superalgebra. Given a homogeneous element $v \in \fg$, we fix the notation $\overline{v}$ to denote the degree of $v$ with respect to the $\mathbb{Z}/2\mathbb{Z}$-grading. If $\overline{v} = 0$, then $v$ is called even, and if $\overline{v} = 1$, then $v$ is called odd. Note that with this setup the even subsuperalgebra $\fg_{\bar 0}$ of $\fg$ may be viewed as a Lie algebra with bracket obtained by restricting the bracket of $\fg$. The Lie superalgebra $\fg = \fg_{\0} \oplus \fg_{\bar 1}$ is called \emph{classical} if there is a connected, reductive algebraic group $G_{\bar 0}$ such that $\Lie(G_{\0}) = \fg_{\0}$, and if there is an action of $G_{\0}$ on $\fg_{\bar 1}$ which differentiates to the adjoint action of $\fg_{\0}$ on $\fg_{\bar 1}$ If, in addition to being classical, $\fg$ has a nondegenerate, invariant, supersymmetric, even bilinear form, then $\fg$ is called basic classical. The basic classical Lie superalgebras were classified by Kac \cite{Kac1}. 

\par Given a Lie superalgebra $\fg$, there is a universal enveloping superalgebra $U(\fg)$ which satisfies a super analog of the PBW theorem for Lie algebras. The category of $\fg$-supermodules has as objects all left $U(\fg)$-supermodules. This means that $\fg$-supermodules are super vector spaces and that the $\fg$-action is compatible with the $\mathbb{Z}/2\mathbb{Z}$-grading. Morphisms between $\fg$-supermodules are even (i.e. degree preserving) morphisms in $\Hom_{\mathbb{C}}(M, M')$ which satisfy $f(xm)=(-1)^{\bar f \bar x}xf(m)$ for all $m \in M$ and all $x \in U(\fg)$. This makes sense as stated only for homogenous elements, and should be extended via linearity in general. Given two $\fg$-supermodules $M, N$, one can use the coproduct and antipode of $U(\fg)$ to give $\fg$-supermodule structures to the vector space tensor product $M \otimes N$, and, when $M$ is finite-dimensional, the contragradient dual $M^{*}$. Denote the category of $\fg$-supermodules as $\fg$-sMod. Because morphisms in $\fg$-sMod are even, it is an abelian category. As convention and notation we consider Lie algebras as Lie superalgebras concentrated in degree $\bar 0$, and we refer to supermodules as modules.

\subsection{The categories $\mathcal{C}_{(\fg, \fg_{\0})}$ and $\mathcal{F}_{(\fg, \fg_{\bar 0})}$ and their stable categories} 
Given a basic classical Lie superalgebra one can consider the $\mathcal{C}_{(\fg, \fg_{\0})}$ which is the category whose objects are $\fg$-modules that, when viewed as modules for the Lie algebra $\fg_{\0}$, are finitely semisimple; i.e., are direct sums of finite-dimensional simple $\fg_{\0}$-modules. The full subcategory of $\mathcal{C}_{(\fg, \fg_{\0})}$ consisting of only finite-dimensional modules is denoted by $\mathcal{F}_{(\fg, \fg_{\bar 0})}$. Both the categories $\mathcal{C}_{(\fg, \fg_{\0})}$ and $\mathcal{F}_{(\fg, \fg_{\bar 0})}$ are abelian categories. Moreover, $\mathcal{C}_{(\fg, \fg_{\0})}$ and $\mathcal{F}_{(\fg, \fg_{\bar 0})}$ are also Frobenius categories. In other words, these categories have enough projective and injective objects, and projectives and injectives coincide. This implies that one can form the stable module categories $\Stab(\mathcal{C}_{(\fg, \fg_{\bar 0})})$ and $\stab(\mathcal{F}_{(\fg, \fg_{\bar 0})})$. Objects in the stable categories are the same as the objects in the underlying abelian categories from which they are formed, but morphisms in the stable categories are equivalence classes of morphisms where two morphisms are considered equivalent if their difference factors through a projective object. The stable categories are triangulated categories, and the tensor product of modules gives $\Stab(\mathcal{C}_{(\fg, \fg_{\bar 0})})$ and $\stab(\mathcal{F}_{(\fg, \fg_{\bar 0})})$ the structure of tensor triangulated categories. Moreover, one has that $\Stab(\mathcal{C}_{(\fg, \fg_{\bar 0})})$ is a rigidly-compactly generated TTC with full subcategory of compact rigid objects $\stab(\mathcal{F}_{(\fg, \fg_{\bar 0})})$. 
Let $\mathcal{C}$ denote the category $\mathcal{C}_{(\fg, \fg_{\bar 0})}$, and let $\mathcal{F}$ denote the category $\mathcal{F}_{(\fg, \fg_{\bar 0})}$. Given modules $M, N$ in $\mathcal{F}$, the group of degree $n$ extensions, $\Ext_{\mathcal{F}}^n(M, N)$ can be realized via relative Lie superalgebra cohomology for the pair $(\fg, \fg_{\bar 0})$:
$$
\Ext_{\mathcal{F}}^n(M, N) \cong \Ext_{(\fg, \fg_{\bar 0})}^n(M, N) \cong \opH^n(\fg, \fg_{\bar 0}; M^* \otimes N). 
$$
There exists an explicit Koszul type resolution that can be used to compute relative Lie superalgebra cohomology. An interesting feature that obtains when considering the relative cohomology ring $\Ext_{\mathcal{F}}^{\bullet}(\mathbb{C}, \mathbb{C}) \cong \opH^{\bullet}(\fg, \fg_{\bar 0}; \mathbb{C})$ is that the cochains are uniformly zero. From this, BKN showed in \cite[Section 2.5]{BKN1} that there is isomorphism 
$$
\opH^{\bullet}(\fg, \fg_{\bar 0}; \mathbb{C}) \cong \text{S}^{\bullet}(\fg_{\bar 1}^*)^{G_{\bar 0}}
$$
of graded rings, and that the relative cohomology is a polynomial algebra \cite[Section 8.11]{BKN1}.

\subsection{Cohomological support varieties} BKN construct, in \cite[Section 4.3]{BKN4}, cohomological support varieties which we need in order to state a key assumption on realization of supports. We briefly review their construction. Let $\fg = \fg_{\bar 0} \oplus \fg_{\bar 1}$ be a classical Lie superalgebra. Set
$$
R:= \opH^{\bullet}(\fg, \fg_{\bar 0}; \mathbb{C}) \cong \text{S}^{\bullet}(\fg_{\bar 1}^*)^{G_{\bar 0}}.
$$
Since $G_{\bar 0}$ is reductive, a classical result of Hilbert gives that $R$ is finitely-generated as a commutative $\mathbb{C}$-algebra. Also, given two modules $M_1$, $M_2$ in $\mathcal{F}_{(\fg, \fg_{\bar 0})}$, $\Ext_{(\fg, \fg_{\bar 0})}^{\bullet}(M_1, M_2)$ is a finitely-generated $R$-module. 
\par The space $X = \Proj(R)$ is a Zariski topological space. Let $\mathcal{X}_{cl}$ denote the collection of closed subsets of $X$. One can construct an assignment $V_{(\fg, \fg_{\bar 0})}: \stab(\mathcal{F}_{(\fg, \fg_{\bar 0})}) \to \mathcal{X}_{cl}$ by setting 
$$
V_{(\fg, \fg_{\bar 0})}(M) := \Proj(R/J_{(\fg, M)})
$$
where $J_{(\fg, M)} = \Ann_R(\Ext_{\mathcal{F}_{(\fg, \fg_{\bar 0})}}^{\bullet}(M, M))$.
\par The assignment $V_{(\fg, \fg_{\bar 0})}(-)$ satisfies properties (2.3.1)-(2.3.4) and (2.3.6) of support data. It is a \emph{pre support datum} in the sense of BKN. Additional assumptions are needed on $\fg$ to guarantee that $V_{(\fg, \fg_{\bar 0})}(-)$ is a support datum. 

\subsection{Detecting subalgebras}
Let $\fg = \fg_{\bar 0} \oplus \fg_{\bar 1}$ be a basic, classical Lie superalgebra. A remarkable contribution of Boe, Kujawa, and Nakano was to prove the existence of important subalgebras called detecting subalgebras $\mathfrak{f} = \mathfrak{f}_{\bar 0} \oplus \mathfrak{f}_{\bar 1} \leq \fg$ which have much easier representation theory than $\fg$ but which nonetheless ``detect" the relative $(\fg, \fg_{\bar 0})$-cohomology. Detecting subalgebras are constructed by considering the action of the algebraic group $G_{\bar 0}$ on $\fg_{\bar 1}$. We recall briefly the parts of their construction we need. 
\par View the set $\fg_{\bar 1}$ as an affine variety with the Zariski topology. A point $x \in \fg_{\bar 1}$ is called regular if the orbit $G_{\bar 0} \cdot x$ has maximum possible dimension, and semisimple if $G_{\0} \cdot x$ is closed in $\fg_{\bar 1}$. The action of $G_{\bar 1}$ is called stable if $\fg_{\bar 1}$ has an open dense subset consisting of semisimple points. If there is an open dense subset of $\fg_{\bar 1}$ such that the stabilizer subgroups of any two points in this set are conjugate subgroups of $G_{\bar 0}$, then the stabilizer of such a point is called a stabilizer in general position. If the action of $G_{\bar 0}$ on $\fg_{\bar 1}$ is stable, then such an open set exists in $\fg_{\bar 1}$. Elements of such an open set are called generic. If the action of $G_{\bar 0}$ on $\fg_{\bar 1}$ is stable then $\fg$ is said to be stable. If $\fg$ is stable, then there is necessarily a generic point $x_{0} \in \mathfrak{g}_{\bar 1}$. Let $H = \text{Stab}_{G_{\bar 0}}(x_0)$ and $N = \text{Norm}_{G_{\bar 0}}(H)$. Set $\mathfrak{f}_{\bar 1} = \mathfrak{g}_{\bar 1}^H$, and $\mathfrak{f}_{\bar 0} = [\mathfrak{f}_{\bar 1}, \mathfrak{f}_{\bar 1}]$. The Lie superalgebra $\mathfrak{f} = \mathfrak{f}_{\bar 0} \oplus \mathfrak{f}_{\bar 1}$ is classical, and is a detecting subalgebra. The sense in which detecting subalgebras determine cohomology is as follows. The inclusion $\mathfrak{f} \leq \mathfrak{g}$ induces a restriction homomorphism $\text{S}^{\bullet}(\mathfrak{f}_{\bar 1}^*) \to \text{S}^{\bullet}(\mathfrak{f}_{\bar 1}^*)$ which induces an isomorphism 
$$
\opH^{\bullet}(\fg, \fg_{\bar 0}; \mathbb{C}) \to \opH^{\bullet}(\mathfrak{f}, \mathfrak{f}_{\bar 0}; \mathbb{C})^N
$$
in cohomology.
\begin{remark}
In \cite{BKN1}, the authors construct two families of detecting subalgebras $\mathfrak{e}$ and $\mathfrak{f}$ of $\mathfrak{g}$. The subalgebras $\mathfrak{e}$ are not considered in this paper. 
\end{remark}


\subsection{Splitting subalgebras}
In order to make use of the Balmer spectrum computations from \cite{BKN1}, and the Nerves-of-Steel result from \cite{HN24}, we need to ensure that our detecting subalgebras satisfy an additional condition. We need that they are so-called \emph{splitting subalgebras}. The idea of splitting subalgebras was introduced by Serganova and Sherman in \cite{SS22}. In the original paper the authors work in the context of the ambient algebraic supergroup, but for our purposes it will be useful to rephrase the definitions somewhat into the context of Lie superalgebras. The following definition is from \cite{HN24} and is equivalent to the original definition. 
\begin{definition}
Let $\fg = \fg_{\bar 0} \oplus \fg_{\bar 1}$ be a classical Lie superalgebra and $G$ be an algebraic supergroup scheme with $\text{Lie }G = \fg$. Moreover, let $Z \leq G$ be a subsupergroup with $\fz = \fz_{\bar 0} \oplus \fz_{\bar 1}$ being classical and $\text{Lie }Z = \fz$. Then $\fz$ is a \emph{splitting subalgebra} if and only if the trivial module $\mathbb{C}$ is a direct summand of $\text{ind}_Z^G\mathbb{C}$. 
\end{definition}

The following theorem summarizes results in \cite[Section 2]{SS22}. The approach presented here is slightly different and  uses the work for BBW parabolic subgroups by D. Grantcharov, N. Grantcharov, Nakano and Wu 
\cite{GGNW21}. 

\begin{theorem}\label{T:splittingprop} Let ${\mathfrak g}$ be a classical Lie superalgebra and ${\mathfrak z}$ be a splitting subalgebra. Let $M$, $N$ be modules in ${\mathcal C}_{({\mathfrak g},{\mathfrak g}_{\0})}$.
\begin{itemize} 
\item[(a)] $R^{j}\operatorname{ind}_{Z}^{G}{\mathbb C}=0$ for $j>0$. 
\item[(b)] $M$ is projective in ${\mathcal C}_{({\mathfrak g},{\mathfrak g}_{\0})}$  if and only if $M$ when restricted to ${\mathfrak z}$ is projective in ${\mathcal C}_{({\mathfrak z},{\mathfrak z}_{\0})}$.  
\item[(c)] For all $n\geq 0$, $\opExt^{n}_{({\mathfrak g},{\mathfrak g}_{\0})}(M,N \otimes \operatorname{ind}_{Z}^{G}{\mathbb C})\cong \opExt^{n}_{({\mathfrak z},{\mathfrak z}_{\0})}(M,N)$. 
\item[(d)] For all $n\geq 0$, the restriction map $\operatorname{res}:\opExt_{({\mathfrak g},{\mathfrak g}_{\0})}^{n}(M,N)\rightarrow \opExt^{n}_{({\mathfrak z},{\mathfrak z}_{\0})}(M,N)$ is injective. 
\end{itemize} 
\end{theorem}

Serganova and Sherman proved that the detecting subalgebra ${\mathfrak f}$ for classical Lie algebras of Type A are splitting subalgebras. \cite[Theorem 1.1]{SS22}.

\section{Stratification for Lie superalgebra representations}

We now investigate the various notions of stratification in the context of Lie superalgebra representations, culminating in a classification of localizing subcategories in certain situations. We set the following as notion and assumptions throughout this section. Let $\mathfrak{g} = \mathfrak{g}_{\bar 0} \oplus \mathfrak{g}_{\bar 1}$ be a classical Lie superalgebra with a splitting, detecting subalgebra $\fz = \fz_{\bar 0} \oplus \fz_{\bar 1} \leq \fg$. The primary example to keep in mind is the general linear Lie superalgebra $\glmn$.

\subsection{BIK stratification for detecting subalgebras} In \cite{HN24}, the author and Nakano employed BIK stratification in order to classify the localizing subcategories for $\Stab(\mathcal{C}_{(\fz, \fz_{\bar 0})})$ where $\fz$ is a detecting subalgebra. We take this opportunity to record the argument, which is based on the analogous result in modular representation theory for elementary abelian two groups in characteristic two. The key fact used is the fact that for the detecting subalgebras $\fz_{\bar 0}$ is a torus, and $[\fz_{\bar 0}, \fz_{\bar 1}] = 0$.  

\begin{theorem}\label{BIKstratified}
The TTC $\Stab(\mathcal{C}_{(\fz, \fz_{\bar 0})})$ is stratified in the sense of BIK by the relative Lie superalgebra cohomology ring $\opH^{\bullet}(\fz, \fz_{\bar 0}; \mathbb{C})$. 
\end{theorem}

\begin{proof}
As notation let $\mathscr{K} = \Stab(\mathcal{C}_{(\fz, \fz_{\bar 0})})$, let $\mathcal{C} = \mathcal{C}_{(\fz, \fz_{\bar 0})}$, and let $R = \opH^{\bullet}(\fz, \fz_{\bar 0}; \mathbb{C})$. Therefore, it remains only to check the minimality condition, but it turns out that it is not convenient to check minimality for $\mathscr{K}$ directly. The idea is to reduce the problem to a different TTC where the result is known via a version of Neeman's theorem. The first step in this direction is to connect $\mathcal{C}$ to the category of supermodules for the superalgebra $\Lambda^\bullet(\fz_{\bar 1})$ which is the exterior algebra on $\fz_{\bar 1}$ viewed as a superalgebra by declaring the generators to be odd. This is done by observing that because $\fz_{\bar 0}$ is a torus which commutes with $\fz$, the weight space decomposition for a $\fz$-supermodule viewed as a module over $\fz_{\bar 0}$ is a decomposition as $\fz$-supermodules. This gives a decomposition of the category $\mathcal{C} = \bigoplus_{\lambda \in \fz_{\bar 0}^*} \mathcal{C}_\lambda$. 
\par The principal block $\mathcal{C}_0$ consists of modules which are annihilated by the ideal $I$ of $U(\fz)$ generated by $U(\fz_{\bar 0})$. Therefore, since $U(\fz)/I \cong \Lambda^\bullet(\fz_{\bar 1}),$ there is an isomorphism of categories $\mathcal{C}_0 \cong \Lambda^{\bullet}(\fz_{\bar 1})\text{-sMod}$, where again $\Lambda^\bullet(\fz_{\bar 1})$ is the exterior algebra on $\fz_{\bar 1}$ viewed as a superalgebra by declaring the generators to be odd. This equivalence passes to an equivalence at the level of the stable module categories: $\mathscr{K}_0 \cong \Stab(\Lambda^{\bullet}(\fz_{\bar 1})\text{-sMod})$. But from this equivalence one sees that it suffices to classify localizing subcategories for the principal block because there is a natural bijection between localizing subcategories for $\mathscr{K}$ and localizing subcategories for $\mathscr{K}_0$. 
\par Next, we observe that a similar problem obtains as the one that occurs for elementary abelian groups in modular representation theory. Namely, the graded endomorphism ring of the unit, $\mathbb{C}$, in $\mathcal{C}$ is not the cohomology ring $R$. Instead, it is an analog of the Tate cohomology ring, which is typically not Noetherian. To get around this problem we instead consider the homotopy category of injectives $\bK(\Inj \mathcal{C}_{(\fz, \fz_{\bar 0})})$. There is an equivalence of triangulated categories between the full subcategory of $\bK(\Inj \mathcal{C}_{(\fz, \fz_{\bar 0})})$ consisting of acyclic complexes $\bK_{\text{ac}}(\Inj \mathcal{C}_{(\fz, \fz_{\bar 0})}) \simeq \Stab(\mathcal{C}_{(\fz, \fz_{\bar 0})})$ and we have the following recollement (c.f. \cite{Kra05}).

\[\begin{tikzcd}
	\bK_{\text{ac}}(\Inj \mathcal{C}_{(\fz, \fz_{\bar 0})}) & \bK(\Inj \mathcal{C}_{(\fz, \fz_{\bar 0})}) & \mathbf{D}(\mathcal{C}_{\mathfrak{g},\mathfrak{g}_{\bar 0}})
	\arrow[from=1-1, to=1-2]
	\arrow[shift left=3, from=1-2, to=1-1]
	\arrow[shift right=3, from=1-2, to=1-1]
	\arrow[from=1-2, to=1-3]
	\arrow[shift right=3, from=1-3, to=1-2]
	\arrow[shift left=3, from=1-3, to=1-2]
\end{tikzcd}\]
But there is an equivalence of TTCs $\bK(\Inj \mathcal{C}_{(\fz, \fz_{\bar 0})}) \simeq \mathbf{D}(S)$, where $S$ is a graded-polynomial superalgebra, and to this one can apply a version of Neeman's theorem. 

\end{proof}
This yields the following, which we state with the most general hypothesis. 
 
\begin{corollary}
Let $\fz = \fz_{\bar 0} \oplus \fz_{\bar 1}$ be a classical Lie superalgebra such that $\fz_{\bar 0}$ is a torus and $[\fz_{\bar 0}, \fz_{\bar 1}] = 0$. There is a bijection between the nonzero $\otimes$-ideal localizing subcategories of the $\Stab(\mathcal{C}_{(\fz, \fz_{\bar 0})})$ and subsets of $\Proj \opH^{\bullet}(\fz, \fz_{\bar 0}; \mathbb{C})$. 
\end{corollary}


\subsection{Nerves-of-Steel for Lie superalgebras}
In order to prove the Nerves-of-Steel Conjecture in Type A, and in order to make use of Theorem \ref{nos}, we need the following assumption on realization of supports. 
\begin{assumption}
Suppose ${\mathfrak g}$ be a classical Lie superalgebra with a splitting subalgebra ${\mathfrak z}\cong {\mathfrak z}_{\bar 0}\oplus {\mathfrak z}_{\bar 1}$. Assume that 
\begin{itemize} 
\item[(i)] ${\mathfrak z}={\mathfrak z}_{\bar 0}\oplus {\mathfrak z}_{\bar 1}$ where ${\mathfrak z}_{\bar 0}$ is a torus and $[{\mathfrak z}_{\bar 0},{\mathfrak z}_{\bar 1}]=0$.
\item[(ii)] Given $W$ an $N$-invariant closed subvariety of $\Proj (\text{S}^{\bullet}(\fz_{\bar 1}^*))$,  there exists $M \in \stab(\mathcal{F}_{(\fg, \fg_{\bar 0})})$ with $V_{(\fz, \fz_{\bar 0})}(M) = W$. 
\end{itemize} 

\end{assumption}\label{assumption}
The following results were proven in \cite{HN24}. Important here is the notion of a tensor triangular field \cite[Definition 1.1]{BKS19}. 
\begin{theorem}\label{noshn}
Let $\fg = \fg_{\0} \oplus \fg_{\bar 1}$ be a classical Lie superalgebra that satisfies Assumption \ref{assumption}. Let $\fz$ denote the splitting, detecting subalgebra so that $\fz = \fz_{\0} \oplus \fz_{\bar 1} \leq \fg$. For $x \in \fz_{\bar 1}$, let $\langle x \rangle$ denote the Lie subsuperalgebra generated by $x$. Then 
\begin{itemize}

\item[(a)] $\Stab(\mathcal{C}_{(\langle x \rangle, \langle x \rangle_{\bar 0})})$ is a tensor triangular field, and the functors 
\begin{equation}
\pi_{x}^{\mathfrak g}: \text{Stab}(C_{({\mathfrak g},{\mathfrak g}_{\0})})\rightarrow \text{Stab}(C_{(\langle x \rangle, \langle x \rangle_{\0})})
\end{equation} 
\begin{equation}
\pi_{x}^{\mathfrak z}: \text{Stab}(C_{({\mathfrak z},{\mathfrak z}_{\0})})\rightarrow \text{Stab}(C_{(\langle x \rangle, \langle x \rangle_{\0})})
\end{equation} 
given by restriction are monoidal exact functors. 

\item[(b)] The Nerves-of-Steel Conjecture holds for each of the TTCs in Part (a). In particular, the comparison map
\begin{equation} 
\phi:\operatorname{Spc}^{\text{h}}(\text{stab}({\mathcal F}_{({\mathfrak g},{\mathfrak g}_{\0})}))\rightarrow \operatorname{Spc}(\text{stab}({\mathcal F}_{({\mathfrak g},{\mathfrak g}_{\0})}))
\end{equation} 
is a bijection. 

\end{itemize}
\end{theorem}

\subsection{Stratification for Type A Lie superalgebras}

Now, we turn our attention to $\Stab(\mathcal{C}_{(\fg, \fg_{\bar 0})})$, where $\fg$ is a classical Lie superalgebra which satisfies Assumption \ref{assumption}. Boe, Kujawa, and Nakano computed the Balmer spectrum to be $N\text{-}\Proj(\text{H}^{\bullet}(\fz, \fz_{\0}; \mathbb{C}))$, and showed that $\Stab(\mathcal{C}_{(\fg, \fg_{\bar 0})})$ is not stratified in the sense of BIK by the cohomology ring $\text{H}^{\bullet}(\fg, \fg_{\0}; \mathbb{C})$. The natural question then is whether or not $\Stab(\mathcal{C}_{(\fg, \fg_{\bar 0})})$ is tt-stratified or h-stratified. That for $\Stab(\mathcal{C}_{(\fg, \fg_{\bar 0})})$ these notions are equivalent and satisfied is Theorem A, which we restate now for convenience. 

\begin{mainthm}
Let $\fg$ be a classical Lie superalgebra with a splitting detecting subalgebra $\fz \leq\fg$ and which satisfies the realization condition of Assumption \ref{assumption}. The tensor triangulated category $\Stab(\mathcal{C}_{(\fg, \fg_{\bar 0})})$ is tt-stratified by the Balmer spectrum $\Spc(\stab(\mathcal{F}_{(\fg, \fg_{\bar 0})}))$, and tt-stratification is equivalent to h-stratification. 
\end{mainthm}


\begin{proof}
First note that by Theorem \ref{noshn}, the Nerves-of-Steel Conjecture holds for $\stab(\mathcal{F}_{(\fg, \fg_{\bar 0})})$. Therefore, since $\Spc(\stab(\mathcal{F}_{(\fg, \fg_{\bar 0})}))$ is weakly Noetherian, the equivalence of tt-stratification and h-stratification of $\Stab(\mathcal{C}_{(\fg, \fg_{\bar 0})})$ follows from Theorem \ref{nos}. 
\par Next, in order to show that the equivalent conditions of tt-stratification and h-stratification hold for $\Stab(\mathcal{C}_{(\fg, \fg_{\bar 0})})$, we show that $\Stab(\mathcal{C}_{(\fg, \fg_{\bar 0})})$ is h-stratified. For this, our tool is Theorem \ref{descent}. For this step, we need to work over an extension field $K$ of $\mathbb{C}$ such that the transcendence degree is larger than the dimension of $\mathbb{z}$ (c.f. \cite[Section 6.2]{HN24} or \cite[Example 3.9]{BC21}). Consider the family of monoidal, exact functors 
\begin{equation}\label{family}
\big(\pi_{x}^{\mathfrak g}: \text{Stab}(C_{({\mathfrak g},{\mathfrak g}_{\0})})\rightarrow \text{Stab}(C_{(\langle x \rangle, \langle x \rangle_{\0})})\big)_{x \in \fz_{\bar 1}}, 
\end{equation}
where $\langle x \rangle$ denotes the Lie subsuperalgebra generated by $x$ and is either a one-dimensional abelian Lie subsuperalgebra, or a subsuperalgebra isomorphic to the queer Lie superalgebra $\mathfrak{q}(1)$. In any case, the relevant fact is that for each $x \in \fz_{\bar 1}$, $\text{Stab}(C_{(\langle x \rangle, \langle x \rangle_{\0})})\big)_{x \in \fz_{\bar 1}}$ is BIK stratified, and therefore tt-stratified and h-stratified. Moreover, the collection of functors in \ref{family} jointly detect when an object of $\Stab(\mathcal{C}_{(\fg, \fg_{\bar 0})})$ is zero. Thus, in order to prove that $\Stab(\mathcal{C}_{(\fg, \fg_{\bar 0})})$ is h-stratified, we need to show that $\Stab(\mathcal{C}_{(\fg, \fg_{\bar 0})})$ is generated by the images of the right adjoints of the functors in \ref{family}. 
 \par It turns out to be convenient to work in the context of the ambient algebraic supergroup scheme. Let $G$ denote the ambient algebraic supergroup scheme such that $\Lie(G) = \fg$, and let $Z \leq G$ be a subsupergroup scheme such that $\Lie(Z) = \fz$. By \cite{GGNW21} the categories $\Stab(\mathcal{C}_{(\fg, \fg_{\bar 0})})$ and $\text{Rep}(G)$ are equivalent. The corresponding family of functors to consider is 
\begin{equation}
\big(\text{res}^G_X: \Stab(\text{Rep}(G)) \to \Stab(\text{Rep}(X)\big)_{x \in Z_{\bar 1}}
\end{equation}
which has right adjoints given by induction: 
\begin{equation}\label{fam2}
\big(\text{ind}^G_X: \Stab(\text{Rep}(X)) \to \Stab(\text{Rep}(G)\big)_{x \in Z_{\bar 1}}. 
\end{equation}
We need to show that the images of the functors in \ref{fam2} generate $\Stab(\text{Rep}(G))$. Notice that by transitivity of induction we have 
$$
\text{ind}^G_X(-) = \text{ind}^G_Z  \text{ind}^Z_X(-). 
$$
Since $Z \leq G$ is a splitting subgroup the image of $\text{ind}^G_Z(-)$ generates $\Stab(\text{Rep}(G))$. To see this let $M$ be a module in $\Stab(\text{Rep}(G))$, because $Z$ is splitting, $M$ is a direct summand of $\text{ind}^G_Z\text{res}^G_Z M$. 
\par It only remains to show that the images $\big(\text{ind}^Z_X(-)\big)_{x \in Z_1}$ generate $\Stab(\text{Rep}(Z))$. To prove this, we again appeal to Theorem \ref{descent}, but in a different way. This time, we use the fact that $\Stab(\text{Rep}(Z))$ being generated by the images of $\big(\text{ind}^Z_X(-)\big)_{x \in Z_1}$ is equivalent to $\Stab(\text{Rep}(Z))$ being tt-stratified. But $\Stab(\text{Rep}(Z))$ is BIK stratified by \ref{BIKstratified} which implies tt-stratification. 
\end{proof}

As a consequence, we obtain the classification of tensor ideal localizing subcategories of $\Stab(\mathcal{C}_{(\fg, \fg_{\bar 0})})$, the content of Corollary B, which we repeat here. 

\begin{maincor}
Let $\fg$ be a classical Lie superalgebra with a splitting detecting subalgebra $\fz \leq\fg$ and which satisfies the realization condition of Assumption \ref{assumption}. There is a bijection between the set of $\otimes$-ideal localizing subcategories of $\Stab(\mathcal{C}_{(\fg, \fg_{\bar 0})})$ and subsets of $N\text{-}\Proj(\text{H}^{\bullet}(\fz, \fz_{\0}; \mathbb{C}))$.
\end{maincor}

\end{document}